\documentclass[12pt]{amsart}
\title{2024 Proof Sketches}
\author{Chad Berner, Noah Giddings, John Herr, Palle Jorgensen}
\date{December 2024}

\usepackage{amsmath,amssymb,cancel,tcolorbox,xcolor}

\usepackage{hyperref}
\usepackage{float}

\newtheorem{theorem}{Theorem}
\newtheorem*{theorem*}{Theorem}
\newtheorem*{corollary*}{Corollary}

\newtheorem{proposition}{Proposition}
\newtheorem{corollary}{Corollary}
\newtheorem{example}{Example}
\newtheorem*{example*}{Example}
\theoremstyle{definition}
\newtheorem{definition}{Definition}
\theoremstyle{remark}

\newtheorem*{remark*}{Remark}
\newtheorem*{remarkn*}{Remark on Notation}
\newtheorem{note}{Note}
\newtheorem*{note*}{Note}

\newcommand{\ds}{\displaystyle}

\newcommand{\abs}[1]{\left\vert#1\right\vert}
\newcommand{\set}[1]{\left\{#1\right\}}
\newcommand{\ip}[1]{\left\langle#1\right\rangle}
\newcommand{\p}[1]{\left(#1\right)}

\begin{document}

\title{Fourier Series for Two Dimensional Singular-Fibered Measures}
\begin{abstract}In this paper we study 2D Fourier expansions for a general class of planar measures $\mu$, generally singular, but assumed compactly supported in $\mathbb{R}^2$. We focus on the following question: When does $L^2(\mu)$ admit a 2D system of Fourier expansions? We offer concrete conditions allowing an affirmative answer to the question for a large class of Borel probability measures, and we present an explicit Fourier duality for these cases. Our 2D Fourier analysis relies on a detailed conditioning-analysis. For a given $\mu$, it is based on the corresponding systems of 1D measures consisting of a marginal measure and associated family of conditional measures computed from $\mu$ by the Rokhlin Disintegration Theorem. Our identified $L^2(\mu)$-Fourier expansions are special in two ways: For our measures $\mu$, the Fourier expansions are generally non-orthogonal, but nonetheless, they lend themselves to algorithmic computations. Second, we further stress that our class of 2D measures $\mu$ considered here go beyond what exists in the literature. In particular, our measures do not require affine iterated function system (IFS) properties, but we do study grid IFS measures in detail and provide some technical criteria guaranteeing their admission of Fourier expansions. Our analyses make use of estimates for the Hausdorff dimensions of the measure supports. An important class of examples addressed in this paper is fractal Bedford-McMullen carpets.
\end{abstract}

\maketitle

\section{Introduction}

The general setting for our present paper is a detailed analysis of generalized (non-orthogonal) Fourier series expansions for Borel measures $\mu$  with compact support in two dimensions. To accomplish this, we shall make use of marginal measures and conditional measures. In more detail, we study these families of one-dimensional measures for each choice of our initial two-dimensional measure $\mu$.

Specifically, given a two-dimensional measure $\mu$ , for each of the two one-dimensional marginal measures computed from $\mu$, we then introduce the respective families of conditional measures, referred to here as fibers or slices. We make this analysis explicit with the use of a corresponding pair of direct integral disintegration decompositions for the Hilbert space $L^2(\mu)$. We address two questions, one of deciding when $L^2(\mu)$ admits a generalized Fourier series expansion, and secondly, in the affirmative, presenting the details for the corresponding Fourier series computations.  The general tool we shall use for this purpose is a disintegration theorem due to Rokhlin. Our use of Rokhlin disintegration (i.e., generalized Bayes rules for conditional measures) serves two purposes.  First, our application of Rokhlin disintegration allows us to give explicit conditions for deciding when $L^2(\mu)$ admits a generalized Fourier series expansion, and second, it allows us to link analysis of the two-dimensional Fourier series to those for the corresponding one-dimensional measures, i.e., the two marginal measures, and the associated family of conditional measures, also called slice-measures.

In our framework for our Fourier expansions, we require the measure to be what we will call \textbf{singular-fibered} that we define using the Rokhlin Disintegration Theorem \cite{Rohlin1949Fundamental,Rohlin1949Decomposition}, which is a generalization of the theorem of Fubini-Tonelli:

\begin{theorem}[Rokhlin Disintegration]
For a Borel probability measure $\mu$ on metric space $A\times B$, there is a unique Borel probability measure $\mu_2$ on $B$, namely, 
$\mu_2=\mu \circ \pi_{B}^{-1}$ where $\pi_{B}: A\times B\to B$ is the projection onto $B$,
and a $\mu_2$-almost-everywhere uniquely determined family of Borel probability measures $\{\rho^{b}\}_{b\in B}$ on $A$ such that the following hold:
\begin{enumerate}
    \item If $f\in L^{1}(\mu)$, then for $\mu_2$-almost-every $b$, $f(a,b)\in L^{1}(\rho^{b})$,\\ and $\int_{A}f(a,b)\,d\rho^{b}\in L^{1}(\mu_2).$
    \item For each $f\in L^{1}(\mu)$,
   $$\int_{A\times B}f\,d\mu=\int_{B}\int_{A}f(a,b)\,d\rho^{b}\,d\mu_2.$$
\end{enumerate}
\end{theorem}

\begin{definition} \label{D:slicesingular}
For a Borel probability measure $\mu$ on $[0,1)^{2}$ by the Rokhlin Disintegration Theorem, there is a Borel probability measure $\mu_{2}$ on $[0,1)$ that we will call the $y$-\textbf{marginal measure} and a family of Borel probability measures $\{\rho^{y}\}$ on $[0,1)$ indexed by $[0,1)$ that we will call the \textbf{slice measures} of $\mu$ such that the conclusions of Rokhlin's theorem hold, namely $\mu(dx\,dy)=\rho^y(dx)\,\mu_2(dy)$. We can also apply the Rokhlin Disintegration Theorem in the other direction and obtain $\mu(dx\,dy)=\rho_x(dy)\,\mu_1(dx)$. 
\end{definition}

These slice measures $\rho_x$ or $\rho^y$ constitute the \textit{fibers} of $\mu$ from a geometric perspective, or from the perspective of probability theory are conditional probability distributions. Conditional distributions and disintegrations are diversely studied in the literature; for example: \cite{Rokh47,Sim12,ChaPol97}.

In \cite{HerWeb2017}, the following result was proved:

\begin{theorem}[Herr \& Weber, 2017]\label{Thm-2016}If $\mu$ is a singular Borel probability measure on $[0,1)$, then any element $f\in L^2(\mu)$ possesses a Fourier series $$f(x)=\sum_{n=0}^{\infty}c_ne^{2\pi inx},$$ where the sum converges in norm.
\end{theorem}

The authors also showed that there exists a Parseval frame $\set{g_n}_{n=0}^{\infty}$ in $L^2(\mu)$, dependent on $\mu$ alone, such that the coefficients $c_n\in\mathbb{C}$ can be computed by $c_n=\ip{f,g_n}_{\mu}$. 

We recall the definition of a frame here:
\begin{definition}
A sequence of vectors $\{f_{n}\}_{n=0}^{\infty}$ in a Hilbert space $H$ is called a \textbf{frame} if there exist $A,B>0$ such that
$$A||f||^{2}\leq \sum_{n=0}^{\infty}|\langle f, f_{n}\rangle|^{2}\leq B||f||^{2}$$ for all $f\in H$.
\end{definition}

A measure $\mu$ on $[0,1)$ is said to be \textbf{spectral} if there exists a set $P\subseteq\mathbb{Z}$ such that the collection $\set{e^{2\pi inx}}_{n\in P}$ is an orthogonal basis in $L^2(\mu)$. For example, in the oft-cited \cite{JorPed98}, Jorgensen and Pedersen demonstrated that the quaternary Cantor measure on $[0,1)$ is spectral with spectrum $P=\set{0,1,4,5,16,17,20,21,\ldots}$. They also showed that the ternary Cantor measure is not spectral. In general, most singular measures are of unknown spectrality. Theorem \ref{Thm-2016}, however, works for any singular probability measure $\mu$ whatsoever, albeit at the cost of the fact that the collection $\set{e^{2\pi inx}}_{n\in\mathbb{Z}}$ is not orthogonal in $L^2(\mu)$ nor even a frame. For a discussion of general results and a wider framework for Fourier analysis, self-similarity, and fractals, we refer the reader to \cite{Jor2018}.

It should be noted that the half-open nature of the interval $[0,1)$ above is necessary, so that $\mu$ can be identified with a measure on the 1-dimensional torus. If the measure $\mu$ had atoms distinctly at both $0$ and $1$, then clearly the result could not hold, since each complex exponential function $e^{2\pi inx}$ is $1$-periodic. By saying a measure $\mu$ is \textbf{supported} on a set $A$, we shall mean only that $A$ is measurable and $\mu(A^\mathcal{C})=0$.

In \cite{HerJorWeb2023} the Rokhlin Disintegration Theorem was used to extend the result of Theorem \ref{Thm-2016} to certain measures on $[0,1)^2$ with the property of being slice-singular. 

\begin{definition}[Singular-Fibered and Slice-Singular Measures]
Let $\mu$ be a Borel probability measure on $[0,1]^2$ with Rokhlin disintegrations \begin{equation*}\mu(dx\,dy)=\rho_x(dy)\,\mu_1(dx)=\rho^y(dx)\,\mu_2(dy).\end{equation*} $\mu$ is said to be $x$-\textbf{singular-fibered} if $\rho_x$ is singular for $\mu_1$-almost-every $x$. If $\mu$ is $x$-singular-fibered and in addition $\mu_1$ is singular, then $\mu$ is said to be \textbf{$x$-slice-singular}. There is a symmetric definition for $y$-singular-fibered and $y$-slice-singular measures. A measure that is slice-singular in both directions is called \textbf{bi-slice-singular}. 
\end{definition}

If a measure $\mu$ on $[0,1]^2$ is slice-singular, then it is also singular. The converse, however, is false. For an easy counterexample, one may simply take the line $y=\frac{1}{2}$ with normalized Lebesgue measure imposed upon it. The resulting measure $\mu$ on the unit square is then singular, but in the $y$ direction the slice $\rho^{1/2}$ is not singular, and in the $x$ direction the marginal $\mu_1$ is Lebesgue. Thus, slice-singular measures constitute a proper subset of the set of singular Borel probability measures on $[0,1]^2$. In \cite{HerJorWeb2023} their special properties were used to obtain the following result:

\begin{theorem}[Herr, Jorgensen, \& Weber, 2023]\label{Thm-23}Suppose $\mu$ is a $y$-slice-singular Borel probability measure on $[0,1)^2$. For any $f\in L^2(\mu)$, $f$ possesses a Fourier series expansion of the form \begin{align}\label{SingMarg}f(x,y)=\sum_{n=0}^{\infty}\sum_{m=0}^{\infty}d_{nm}e^{2\pi i(nx+my)}.\end{align} The series converges conditionally in norm.
\end{theorem}

Again, the authors prescribe how the coefficients can be computed. The theorem applies similarly if the measure is $x$-slice-singular. 

Although formulated so that $\mu$ is supported on $[0,1)^2$, it is easy to see that the result would be equally valid if $\mu$ were supported on any of the sets $[0,1)^2$, $[0,1)\times(0,1]$, $(0,1]\times(0,1]$, or $(0,1]^2$, or in some cases, if each individual slice measure and the marginal measure were supported on various half-open intervals. To strike a balance between simplicity and flexibility, we will frame most of our results for $[0,1)^2$. The reader shall be alert to how certain results can be easily generalized to more situations.

In Lemma 2 of \cite{HerJorWeb2023} and its proof, the following simple facts were established, the truth of which the reader will easily see:
\begin{proposition}\label{prop-simple}Let $\mu$ be a measure on $[0,1]^2$ and $\mu_1$ and $\mu_2$ its $x$- and $y$-marginal measures. Let $\mu(dx\,dy)=\rho_x(dy)\,d\mu_1(dx)=\rho^y(dx)\,\mu_2(dy)$ be Rokhlin disintegrations of $\mu$. If $\mu_2$ is singular, then $\rho_x$ is singular for $\mu_1$-almost-every $x$. Likewise, if $\mu_1$ is singular, then $\rho^y$ is singular for $\mu_2$-almost-every $y$. Consequently, if both $\mu_1$ and $\mu_2$ are singular, then $\mu$ is bi-slice-singular.
\end{proposition}

This result thus identifies a simple condition guaranteeing a measure is slice-singular: If both marginals are singular, the measure is slice-singular in both directions and hence will admit Fourier series as in $\eqref{SingMarg}$. Of course, a measure need not be bi-slice-singular to be slice-singular. Our goal in this paper is to establish that a much larger class of measures admit Fourier series.

Throughout the paper, we shall prove the singularity of measures by determining the Hausdorff dimension of their support. A key technical step in our analysis is concerned with determining the Hausdorff dimensions of the intersections of lines with special planar sets. While there is an extensive literature dealing with this question, for our present purposes we shall make use of a classical approach known as the Marstrand Slicing Theorem. The theorem, while taking different forms in the literature, refers to a central result in fractal geometry, essentially stating that, for a planar set with Hausdorff dimension strictly greater than one, almost all lines intersecting the set (by slicing) will do so in a set of dimension one less than the original set's dimension, meaning that the ``slices" of the set formed by lines have a dimension one lower; this theorem is primarily attributed to John Marstrand \cite{Mar54}, and it is often discussed in the wider context of Kenneth Falconer's work on fractal geometry and projections. In this paper, we shall make use of a more general variant that appears in some lectures by Hochman \cite{Hoch14} and a book by Bishop and Peres \cite{BisPer2017}. It fits our present needs, but we also wish to stress that there is an extensive literature dealing with the Hausdorff dimension of slices or projections of planar sets. For the reader's benefit we call attention to the following small sample: \cite{FalFraJin15,BarFerSim12,FalMat16,Aritro24,FalFraShm21,BurFalFra21}.

An unfamiliar reader may find the Hausdorff dimension defined and its essential properties noted in \cite{Hoch14}, \cite{BisPer2017}, \cite{Falc03}, and in innumerable other texts. We shall denote the Hausdorff dimension of a set $X$ by $\dim_{H}(X)$. The basic idea is that a set that has some fractal nature to it may be designated some non-integral dimension. Of special importance is the fact that if $\dim_{H}(X)<n$, then $\lambda^n(X)=0$, where $\lambda^n$ is the $n$-dimensional Lebesgue measure. The converse does not hold: There are sets for which $\lambda^n(X)=0$ and yet $\dim_{H}(X)=n$. There are many alternatives to the Hausdorff dimension, such as the Minkowski/box dimension (c.f. \cite{Falc03,BisPer2017,FalFraKae23,Olsen19}) and the packing dimension (c.f. \cite{BisPer2017,Olsen15}), but perhaps none rise to the same level of popularity.

Our paper is organized as follows: We first expand Theorem \ref{Thm-23} to the case where the marginal measure is similar to Lebesgue measure instead of singular, which will still admit Fourier series, albeit with a bi-infinite inner sum. We then show that a large class of measures satisfy the sufficient conditions for having such series. Afterward, we return to the question of when a measure is slice-singular, since  such measures have Fourier series which, not involving a bi-infinite summation and additional terms, are more desirable. Section 3 is devoted to grid IFS measures.

\section{Main Results}

Although Theorem \ref{Thm-23} is formulated in terms of slice-singularity, the spectrality of the Lebesgue measure means the result can be expanded easily to the case where the marginal measure is similar to Lebesgue measure. Our first main result illustrates that there are more singular-fibered measures possessing these types of Fourier expansions than just the slice-singular ones. Before we prove this result, we make our definition of similar-to-Lebesgue-measure precise:
\begin{definition}

Let $\mu$ be an absolutely continuous finite Borel measure on $[0,1)$ with Radon-Nikodym derivative $h$. We say $\mu$ is \textbf{essentially equivalent to Lebesgue measure} if $h,\frac{1}{h}\in L^{\infty}(\mu)$.
\end{definition}

\begin{theorem}\label{LebMargResult}Let $\mu$ be a Borel probability measure on $[0,1)^2$, and let $\mu(dx\,dy)=\rho^y(dx)\,\mu_2(dy)$ be a Rokhlin disintegration of $\mu$. Suppose $\mu$ is y-singular-fibered, and $\mu_2$ is essentially equivalent to Lebesgue measure. For any $f\in L^2(\mu)$, $f$ possesses a Fourier series expansion of the form \begin{align}\label{LebMarg}f(x,y)=\sum_{n=0}^{\infty}\sum_{m=-\infty}^{\infty}d_{nm}e^{2\pi i(nx+my)}
\end{align}
The sum converges conditionally in the norm of $L^2(\mu).$

Furthermore for each $n,m$, the mapping of $f\mapsto d_{nm}$ is a continuous linear functional, and the mapping
    $f\mapsto \{d_{nm}\}$ into $\ell^{2}(\mathbb{N}_0\times \mathbb{Z})$ has a bounded left inverse.
\end{theorem}

Our proof follows the ideas of the proof of Theorem \ref{Thm-23} made in \cite{HerJorWeb2023}, which still works when $\mu_2$ is assumed to be equivalent to Lebesgue measure rather than singular, with only a few very simple adjustments needed. Our discussion relies on a familiarity with that proof, which is far too long to recapitulate here, but which we invite the reader to examine. References to theorems and propositions within our proof of Theorem \ref{LebMargResult} are in the context of \cite{HerJorWeb2023}.

At no point does the proof in \cite{HerJorWeb2023} depend on $\mu_2$ being a singular measure. It requires only that $\mu_2$ be the marginal measure from the Rokhlin Disintegration Theorem for the purpose of ensuring that the operator $$R_n:L^2(\mu)\to L^2(\mu_2):f(x,y)\mapsto\p{y\mapsto\int f(x,y)e^{-2\pi inx}\rho^y(dx)}$$ is well-defined and that the $R_n^\ast$'s are isometries. If $\mu_2$ is equivalent to Lebesgue measure rather than singular, then $\set{e^{2\pi imy}}_{m\in\mathbb{N}_0}$ will not be complete in $L^2(\mu_2)$, but $\set{e^{2\pi imy}}_{m\in\mathbb{Z}}$ will be a basis. We still will have that $R_n^\ast R_n=P_n$ is the orthogonal projection onto \begin{align*}\mathcal{M}_n=\set{e^{2\pi inx}g(y):g\in L^2(\mu_2)},\end{align*}
but this time it is seen that \begin{align*}\mathcal{M}_n=\overline{span}\set{e^{2\pi i(nx+my)}:m\in\mathbb{Z}},\end{align*}
where $\mathbb{N}_0$ has been replaced by $\mathbb{Z}$. 

\begin{proof}[Proof of Theorem \ref{LebMargResult}]

Theorem 9.3 in \cite{HerJorWeb2023} still shows that $\set{R_n}_{n\in\mathbb{N}_0}$ is an effective sequence of operators. Because $\mu_2$ is the marginal measure, we still have that \begin{align*}R_n^\ast:g(y)\mapsto e^{2\pi inx}g(y).
\end{align*} are the adjoint operators. Proposition 9.2 still shows that  \begin{align*}[G_nf](y)=\int f(x,y)\overline{g_n^{(y)}(x)}\rho^y(dx).
\end{align*}
is the auxiliary operator sequence of $\set{R_n}_{n\in\mathbb{N}_0}$, and then $\sum_{n=0}^{\infty}R_n^\ast G_n$ converges to the identity in the strong operator topology.

Now if $\mu_2$ is essentially equivalent to Lebesgue measure, for any $g\in L^2(\mu_2)$, we have that \begin{align*}g=\sum_{m\in\mathbb{Z}}\ip{g,\frac{e_m}{h}}_{\mu_{2}}e_m\end{align*} where $h$ is the Radon-Nikodym derivative of $\mu_{2}$ since $\{e_m\}_{m\in \mathbb{Z}}$ is a frame in $L^{2}(\mu_{2})$.
This convergence occurs in the norm of $L^2(\mu_2)$, but it also occurs in the norm of $L^2(\mu)$ if we regard the functions as being in $L^2(\mu)$, because the map $L^2(\mu_2)\to L^2(\mu)$ given by $g(y)\mapsto\p{(x,y)\mapsto g(y)}$ is an isometry.
Hence, for any $f\in L^2(\mu)$, we have 
\begin{align*}f=&\sum_{n=0}^{\infty}R_n^\ast G_n f\\
=&\sum_{n=0}^{\infty}\p{\sum_{m=-\infty}^{\infty}\ip{[G_nf](y),\frac{e^{2\pi imy}}{h(y)}}_{\mu_2}e^{2\pi imy}}e^{2\pi inx}\\
=&\sum_{n=0}^{\infty}\p{\sum_{m=-\infty}^{\infty}\p{\int f(x,y)\overline{g_n^{(y)}(x)\frac{e^{2\pi imy}}{h(y)}}\,\mu(dx\,dy)}e^{2\pi imy}}e^{2\pi inx}\\
=&\sum_{n=0}^{\infty}\sum_{m=-\infty}^{\infty}\ip{f(x,y),g_n^{(y)}\frac{e_m(y)}{h(y)}}_{\mu}e^{2\pi i(nx+my)}.
\end{align*}
By taking $d_{nm}=\ip{f(x,y),g_n^{(y)}\frac{e_m(y)}{h(y)}}_{\mu}$, the theorem is proved.

Finally, let $f\in L^{2}(\mu)$. By Corollary 9.3 in \cite{HerJorWeb2023}, consider the following:
\begin{equation}\label{pf}
\begin{split}
&\sum_{n}\sum_{m}\left|\ip{f(x,y),g_n^{(y)}\frac{e_m(y)}{h(y)}}_{\mu} \right|^{2} \\
=&\sum_{n}\sum_{m} \left| \left \langle \left \langle f(x,y),g_n^{(y)} \right \rangle_{\rho^{(y)}} ,\frac{e_m(y)}{h(y)} \right \rangle_{\mu_{2}}  \right|^{2}  \\
 \leq& B \sum_{n} \left \|\left \langle f(x,y),g_n^{(y)} \right \rangle_{\rho^{(y)}} \right \|_{\mu_{2}}^{2}\\
 =&B ||f||_{\mu}^{2}
\end{split}
\end{equation}
where $B$ is the upper frame bound for $\left \{ \frac{e_{m}(y)}{h(y)} \right \}$.

A similar calculation shows
$$\sum_{n}\sum_{m}\left|\ip{f(x,y),g_n^{(y)}\frac{e_m(y)}{h(y)}}_{\mu} \right|^{2} \geq A ||f||_{\mu}^{2}$$
where $A$ is the lower frame bound for $\left \{ \frac{e_{m}(y)}{h(y)} \right \}$.
\end{proof}

\begin{theorem}\label{SingWithLebMarg}Suppose $\mu$ is a singular Borel probability measure on $[0,1)^2$ with a marginal measure that is essentially equivalent to Lebesgue measure. Then $L^2(\mu)$ admits Fourier series of form \eqref{LebMarg}.
\end{theorem}

\begin{proof}Let $\mu$ be a singular Borel probability measure on $[0,1)^2$. Without loss of generality, suppose that the $y$-marginal measure $\mu_2$ of $\mu$ is essentially equivalent to Lebesgue measure, and let $\mu=\rho^y(dx)\,\mu_2(dy)$ be its Rokhlin disintegration. Let $S\subseteq[0,1)^2$ be a set of Lebesgue measure 0 on which $\mu$ is supported.  Let $S_y=\set{x:(x,y)\in S}$ denote a cross-section of $S$ taken at $y$. Let $\chi_S$ be the indicator function of $S$. We compute that
\begin{align*}0=&\lambda(S)\\
=&\int_{0}^{1}\int_{0}^{1}\chi_S(x,y)\,dx\,dy.
\end{align*}
It follows that $\lambda^1(S_y)=\int_{0}^{1}\chi_S(x,y)\,dx=0$ for Lebesgue-almost-every $y$ and hence for $\mu_{2}$ almost every $y$ since $\mu_{2}$ is absolutely continuous.

Now assume, for the sake of contradiction, that there exists a set $A\subseteq[0,1)$ of $\mu_{2}$ positive measure such that $\rho^y(S_y^\mathcal{C})>0$ for all $y\in A$. Define $X=([0,1)\times A)\setminus S.$ Then $X$ is measurable and $\mu(X)=0$. However, we also have
\begin{align*}\mu(X)=&\int\chi_X(x,y)\mu(dx\,dy)\\
=&\int_A\int_{S_y^\mathcal{C}}\rho^y(dx)\,\mu_2(dy)\\
=&\int_{A}\rho^y(S_y^\mathcal{C})\,\mu_2(dy)\\
>&0,
\end{align*}
which is a contradiction. 

Hence, $\rho^y(S_y^\mathcal{C})=0$ for $\mu_{2}$-almost-every $y$. It follows for $\mu_{2}$-almost-every $y$, $\rho^y$ is supported on a set of Lebesgue measure 0 and is therefore singular. Since $\mu_2$ is essentially equivalent to Lebesgue measure, the result then follows from Theorem \ref{LebMargResult}.
\end{proof}

We collect our results:

\begin{corollary}\label{Cor_Collect}Let $\mu$ be a measure on the half-open unit square. 
\begin{enumerate}\item If the marginal measures of $\mu$ are both singular, then $\mu$ is bi-slice-singular and hence $L^2(\mu)$ admits Fourier series of type $\eqref{SingMarg}$.
\item If the marginal measures of $\mu$ are singular and essentially equivalent to Lebesgue measure, then $L^2(\mu)$ admits Fourier series of type $\eqref{LebMarg}$.
\item If the marginal measures of $\mu$ are both essentially equivalent to Lebesgue measure and $\mu$ itself is singular, then $L^2(\mu)$ admits Fourier series of type $\eqref{LebMarg}$.
\end{enumerate}
\end{corollary}

\begin{proof}(1) is a simple application of Proposition \ref{prop-simple} and Theorem \ref{Thm-23}, done in \cite{HerJorWeb2023}, and (3) is an immediate consequence of Theorem \ref{SingWithLebMarg}. So suppose that $\mu$ is a measure on $[0,1)^2$, and without loss of generality, suppose its $x$-marginal-measure $\mu_1$ is singular. Then $\mu_1$ is supported on a set $A$ of Lebesgue measure $0$, and $\mu$ itself is supported on the set $A\times[0,1)$, which is a set of Lebesgue measure $0$. Hence $\mu$ is singular, and again we may apply Theorem $\ref{SingWithLebMarg}$.
\end{proof}

Even if we have Fourier series of type $\eqref{LebMarg}$, it may be more desirable to have a series of type $\eqref{SingMarg}$. In the situation in which both marginals are singular, this is guaranteed to happen. Otherwise, so long as at least one marginal is singular, it could be the case that the measure is slice-singular, in which case we get type $\eqref{SingMarg}$ Fourier series by Theorem $\ref{Thm-23}$.
The following Theorem \ref{Thm-A} helps address this question for the case of Frostman measures. To prove it, we draw on the following generalized Marstrand slicing theorem, which can be found in \cite{Hoch14} and \cite{BisPer2017}:

\begin{theorem}[Generalized Marstrand Slicing Theorem]
Let $\mu$ be a probability measure on $[0,1]$ such that there exist constants $C,\alpha\geq0$ where 
$$\mu(I)\leq C|I|^{\alpha}$$ for all intervals $I\subseteq [0,1]$. If $S\subseteq[0,1]^2$ is compact,
then 
$$\dim_{H}(S_{x})\leq \max\set{0,\dim_{H}(S)-\alpha}$$ for $\mu$-almost-every-$x$,
where $S_x=\set{y:(x,y)\in S}$ is the cross-section of $S$ taken at $x$.
\end{theorem}

\begin{theorem}\label{Thm-A}
Let $\mu$ be a singular Borel probability measure on the half-open unit square, and let $S\subseteq[0,1]^2$ be a compact set on which $\mu$ is supported. Let the $x$-marginal measure $\mu_1$ of $\mu$ be singular, and suppose that there exist constants $C$ and $\alpha$ such that $\mu_1(I)\leq C\abs{I}^\alpha$ for all intervals $I\subseteq[0,1]$. Suppose also that $\dim_{H}(S)<\alpha+1$. Let $\mu=\rho_x(dy)\,\mu_1(dx)$ be the Rokhlin disintegration of $\mu$. Then  $\rho_x$ is singular for $\mu_1$-almost-every-$x$, and $L^2(\mu)$ admits Fourier expansions of type $\eqref{SingMarg}$.
\end{theorem}

\begin{proof} Let $S_x=\set{y:(x,y)\in S}$ be the cross-section of $S$ taken at $x$. Then $S_x$ is Borel. We may assume without loss of generality that each slice measure $\rho_{x}$ from the Rokhlin disintegration is supported on $S_x$. 

By the Generalized Marstrand Slicing Theorem, we have that $$\dim_{H}(S_x)\leq\max\set{0,\dim_H(S)-\alpha}<1$$ for $\mu_1$-almost-every-$x$. It follows that $\mu_1$-almost-every $\rho_x$ is supported on a set of Lebesgue measure 0 and hence singular. 

Thus, $\mu$ is $x$-slice-singular, and by assumption of the support of $\mu$, the marginal measure and the slice measures are measures supported on $[0,1)$.Hence by Theorem \ref{Thm-23}, $L^2(\mu)$ admits Fourier series of type $\eqref{SingMarg}$.
\end{proof}

\begin{note}The same result obtains if we use the $y$-marginal measure $\mu_2$.
\end{note}

\begin{note}If $\mu_1$ has an atomic part, the condition $\mu_1(I)\leq C\abs{I}^\alpha$ can only be satisfied by $\alpha=0$.
\end{note}

\begin{example}If $\mu_1$ is the ternary Cantor measure, then we have $\mu_1(I_\sigma)=\abs{I_\sigma}^{\ln(2)/\ln(3)}$ for any $3$-adic interval $I_\sigma$ that intersects the Cantor set, and so $\mu_1(I_\sigma)\leq\abs{I_\sigma}^{\ln(2)/\ln(3)}$ for any $3$-adic interval. We now use a standard argument: Let $I\subseteq[0,1)$ be any interval, and let $k$ be the smallest integer such that $\frac{1}{3^k}\leq\abs{I}$. Then $I$ can be covered using at most $(3+1)$ $3$-adic intervals of length $\frac{1}{3^k}$. It follows that $\mu_1(I)\leq 4\abs{I}^{\ln(2)/\ln(3)}$. Thus, the theorem will apply so long as $\dim_{H}(S)<\frac{\ln(2)}{\ln(3)}+1=\frac{\ln(6)}{\ln(3)}$.
\end{example}

\begin{example}Consider the Sierpinski triangle $K$, which is a Bedford-McMullen carpet generated by $T=\begin{bmatrix}1&0\\1&1\end{bmatrix}$ (see Definition \ref{Def:GridIFS}). We show the first six iterations of this IFS applied to the unit square:
\begin{figure}[H]\includegraphics[scale=.55]{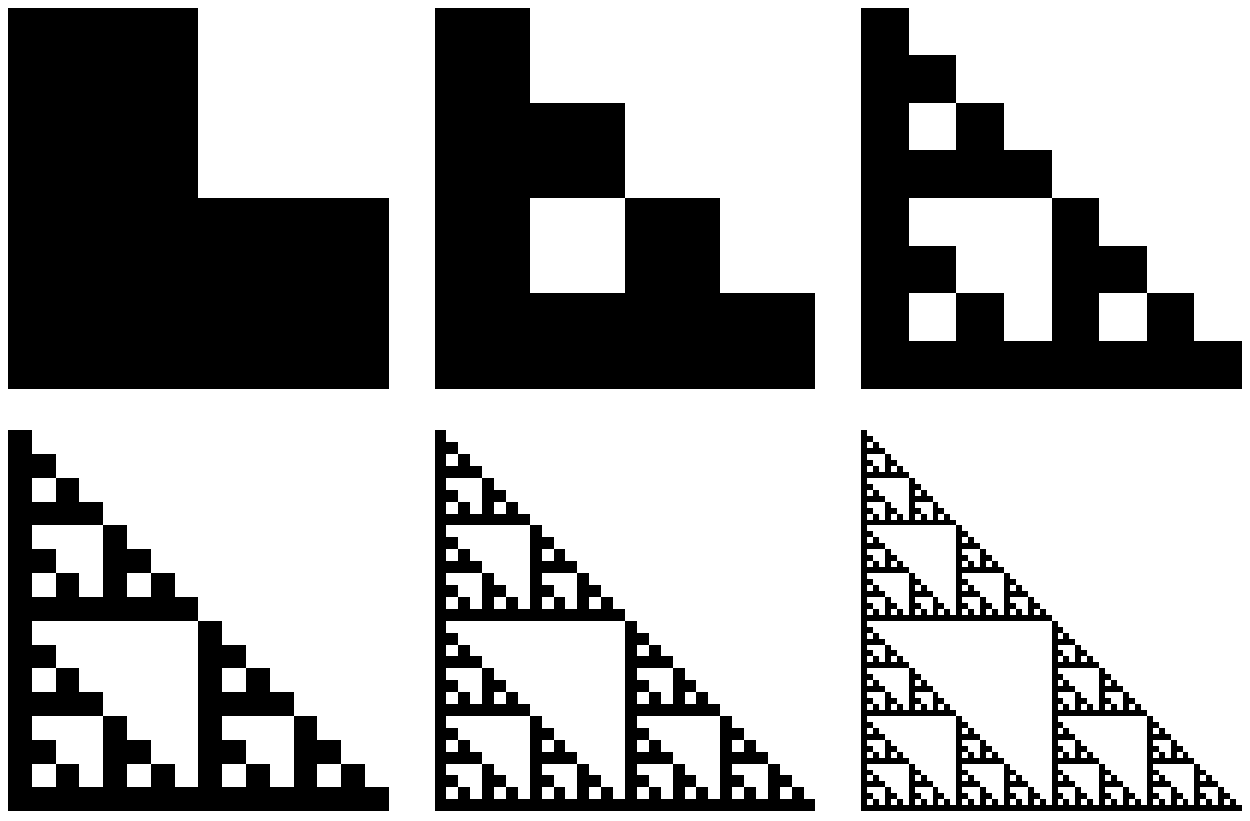}\caption{The Sierpinski Triangle}
\end{figure}

Let $\mu$ be the invariant measure associated to this IFS with equal weights, and let $\mu_1$ be the $x$-marginal measure. On a dyadic interval $I_\sigma$ of length $\frac{1}{2^k}$, we must have $$\p{\frac{1}{3}}^k\leq\mu_1(I_\sigma)\leq\p{\frac{2}{3}}^k=\abs{I_\sigma}^{\frac{\ln(3)}{\ln(2)}-1}.$$
Following the standard argument, it follows that for any interval $I$, $$\mu_1(I)\leq 3\abs{I}^{\frac{\ln(3)}{\ln(2)}-1}.$$

Thus, we may take $\alpha=\frac{\ln(3)}{\ln(2)}-1$. Since the Sierpinski triangle has Hausdorff dimension $\frac{\ln(3)}{\ln(2)}$, we fail to have $\dim_H{K}<\alpha+1$, and the theorem does not apply.

However, if the weights are changed to $p_{0,0}=\frac{1}{4}$, $p_{0,1}=\frac{1}{4}$, and $p_{1,0}=\frac{1}{2}$, then we can reach the same conclusion anyway by the upcoming Corollary \ref{Cor:SquareCarpet}.
\end{example}

\section{Results for Grid IFS's}\label{Sec:Grid}

As a reminder, the following Hutchinson's Theorem \cite{Hutchinson1981Fractals} establishes that iterated function systems (IFS) of contractions give rise to Borel probability measures.

\begin{theorem}[Hutchinson]\label{Thm:Hutchinson}
Let $(X,d)$ be a complete metric space and $\phi=\{\phi_{1},\dots ,\phi_{N}\}$ be a finite set of contractions on $X$. Then there exists a unique compact set $K$ such that
$$K=\bigcup_{k=1}^{N}\phi_{k}(K).$$
Furthermore, $K$ is the closure of the set of fixed points of finite compositions of members of $\phi$.

Moreover, given $w_{1},\dots, w_{N}>0$ such that $\sum w_{k}=1$, there exists a unique Borel probability measure $\mu$, called the invariant measure, that is supported on $K$ and such that
$$\int f\,d\mu=\sum_{k=1}^{N}w_{k}\int f\circ \phi_{k}\, d\mu$$
holds for every continuous function $f$.
\end{theorem}

We will also use the following result of Kakutani \cite{Kakutani1948Equivalence} which relates infinite product measures generated with equivalent components:

\begin{theorem}[Kakutani, 1948]
For each $n\in \mathbb{N}$, let $\mu_{n}$ and $v_{n}$ be measures on $\mathbb{R}$ that are mutually absolutely continuous. Let $\mu=\times_{n=1}^{\infty} \mu_{n}$ and $v=\times_{n=1}^{\infty}v_{n}$. Define $$\rho(v_n,\mu_n)=\int_{\mathbb{R}}\sqrt{\frac{d\mu_n}{dv_n}}\,dv_n.$$
If $$\prod_{n=1}^{\infty}\rho(v_n,\mu_n)>0,$$ then $\mu$ and $v$ are mutually absolutely continuous. Otherwise, if $$\prod_{n=1}^{\infty}\rho(v_n,\mu_n)=0,$$ then $\mu$ and $v$ are mutually singular.
\end{theorem}

When applied to the case of invariant measures of IFS's, the Kakutani Theorem yields the following known result:
\begin{theorem}[Kakutani Dichotomy Theorem for IFS Measures]\label{Thm:KakutaniIFS}
Suppose that $\{\phi_{1},\dots ,\phi_{N}\}$ are contractions acting on a metric space $(X,d)$, and let $\mu$ and $\tilde{\mu}$ be the invariant measures arising from these contractions with positive weights $\{w_{1},\dots ,w_{N}\}$ and $\{\tilde{w}_{1},\dots ,\tilde{w}_{N}\}$ respectively. Then exactly one of the following holds:
\begin{enumerate}
    \item $\mu=\tilde{\mu}$, which is equivalent to $w_{j}=\tilde{w_{j}}$ for all $j$.
    \item $\mu$ and $\tilde{\mu}$ are mutually singular.
\end{enumerate}
\end{theorem}

Thus, if the weights given to an IFS are different, the induced invariant measures are mutually singular.

In this section, we will focus on IFS's that can be thought of as dividing the unit square into an $m\times n$ grid of rectangles, keeping some of them, and then iteratively copying the remaining image into each one. More precisely, we have the following:

\begin{definition}\label{Def:GridIFS}Let $T=(t_{ij})$ be an $m\times n$ matrix with indexing as described below and whose entries consist entirely of $0$'s and $1$'s. Then the IFS whose contractions are
$$\set{\tau_{ij}(x,y)=\p{\frac{x}{n}+\frac{i}{n},\frac{y}{m}+\frac{j}{m}}:t_{ij}=1}$$
is called a \textbf{grid IFS}, and we refer to its compact attractor as a \textbf{Bedford-McMullen carpet}.
\end{definition}

We will assume that $T$ is indexed beginning at 0, with rows indexed by $0\leq j\leq m-1$ bottom to top, and columns indexed by $0\leq i\leq n-1$ left to right. While nonstandard for matrices, these conventions have a number of conveniences in our situation, among which are that the visual appearance of the matrix will match the visual appearance of the image of the induced IFS applied to the unit square, with the 1's of the matrix matching up with the cells kept by the IFS. Furthermore, the indices will correspond to the digits of base-$n$ and base-$m$ expansions and also the shifts of the IFS.

With a grid IFS in place, we may create a set of weights $p_{ij}$ for the contractions, and summarize them with an $m\times n$ matrix $P=(p_{ij})$, indexed as $T$ is, where $p_{ij}=0$ if $t_{ij}=0$, $p_{ij}>0$ if $t_{ij}=1$, and $\sum_{i,j}p_{ij}=1$. By Theorem \ref{Thm:Hutchinson}, there then exists an invariant measure $\mu$ supported on the attractor of the IFS such that \begin{equation}\int f\,d\mu=\sum_{t_{ij}=1}p_{ij}\int f\circ\tau_{ij}\,d\mu\end{equation} for every continuous function $f$.

The marginal measures of grid IFS's are also IFS-induced measures. The $x$-marginal measure $\mu_1$ is the invariant measure of the IFS
\begin{equation}\set{\tilde{\tau}_i(x):=\frac{x}{n}+\frac{i}{n}:0\leq i<n,t_{ij}=1\text{ for some }j}
\end{equation}
with weights $w_i=\sum_{j=0}^{m-1}p_{ij}$. A similar result holds for $\mu_2$. One may think of this as projecting the IFS to obtain an IFS in one dimension.

Some nuance must be employed when interpreting the weights $p_{ij}$ and their relationship to the invariant measure $\mu$. Properly speaking, these weights do not correspond to the $\mu$-measure of the cells, but to a measure placed on a space of symbolic representations of coordinates $(x,y)$. We refer the reader to Appendix \ref{App:SymbolicRepresentations} for a treatise. We will use the idea behind this representation of $\mu$ in the proof of Theorem \ref{Thm-fullgrid}. 

\begin{corollary}[Rectangular Grid IFS's]\label{Cor-A}Let $T$ be an $m\times n$ matrix, with $2\leq m\leq n$, and let $K$ be the corresponding Bedford-McMullen carpet, the attractor of the IFS $$\set{\tau_{ij}(x,y)=\p{\frac{x}{n}+\frac{i}{n},\frac{y}{m}+\frac{j}{m}}:t_{ij}=1}.$$ Let $p_{ij}$ be the weights of this IFS, which are positive. Let $\mu$ be the invariant measure of this IFS, with Rokhlin disintegrations $\mu=\rho_x(dy)\,\mu_1(dx)=\rho^y(dx)\,\mu_2(dy)$. Let \begin{align*}\gamma_1=\max_{i}\set{\sum_{\substack{j\\t_{ij}=1}}p_{ij}}\\
\gamma_2=\max_{j}\set{\sum_{\substack{i\\t_{ij}=1}}p_{ij}}\end{align*} be the max total weight of the IFS in any single column or row, respectively.

Finally, let $r_j$ denote the number of 1's in the $j$th row of $T$, and define 
\begin{align*}\xi_1&=\p{\sum_{j=0}^{m-1}r_j^{\frac{\ln(m)}{\ln(n)}}}^{\frac{\ln(n)}{\ln(m)}},\\
\xi_2&=\p{\sum_{j=0}^{m-1}r_j^{\frac{\ln(m)}{\ln(n)}}}.
\end{align*}

Then we have the following:
\begin{enumerate}
\item If $\gamma_1<\frac{n}{\xi_1}$, then $\mu_1$-almost every slice measure $\rho_x$ is singular, and $L^2(\mu)$ admits Fourier series of type $\eqref{SingMarg}$ or $\eqref{LebMarg}$ according to whether $\mu_1$ is singular or Lebesgue.
\item If $\gamma_2<\frac{m}{\xi_2}$, then $\mu_2$-almost every slice measure $\rho^y$ is singular, and $L^2(\mu)$ admits Fourier series of type $\eqref{SingMarg}$ or $\eqref{LebMarg}$ according to whether $\mu_2$ is singular or Lebesgue.
\end{enumerate}
\end{corollary}

\begin{proof}By \cite{McM1984}, we have that $$\dim_{H}(K)=\frac{\ln\p{\sum_{j=0}^{m-1}r_j^{\ln(m)/\ln(n)}}}{\ln{m}}.$$

Let $\mu_1$ be the $x$-marginal measure of $\mu$. If $T$ contains a column of all 0's, then it follows by projecting the IFS onto the $x$-axis that the Hausdorff dimension of the projection of $K$ onto the $x$-axis is less than 1 and hence has Lebesgue measure zero. Thus $\mu_1$ is a singular measure. Otherwise, we may apply the Kakutani Dichotomy Theorem, and we have that $\mu_1$ is either Lebesgue measure itself or singular with respect to Lebesgue measure. By a symmetric argument, we have that the $y$-mariginal measure $\mu_2$ also is either Lebesgue measure or singular.

Now, neither $\mu_1$ nor $\mu_2$ can have atoms at both $0$ and $1$, for that would require two columns (respectively, rows) to have weight 1. We may thus assume that $\mu$ is supported on one of the following four sets: $[0,1)^2$, $[0,1)\times(0,1]$, $(0,1]\times[0,1)$, or $(0,1]^2$.

For any $n$-adic interval $I_\sigma$ of length $\frac{1}{n^k}$, we have $\mu_1(I_\sigma)\leq\gamma_1^k=\abs{I_\sigma}^{-\frac{\ln(\gamma_1)}{\ln(n)}}$. As usual, this implies $\mu_1(I)\leq(n+1)\abs{I}^\alpha$ for any interval $I$, where $\alpha=-\frac{\ln(\gamma_1)}{\ln(n)}.$ Following similar logic, if $\mu_2$ is the $y$-marginal measure of $\mu$, then for any interval $I$, $\mu_2(I)\leq(m+1)\abs{I}^{\beta}$ where $\beta=-\frac{\ln(\gamma_2)}{\ln(m)}.$ 

Now we will establish that $\frac{1}{\xi_1}\leq\gamma_1$. Let $N$ be the number of 1's in $T$. Note that $N=\sum_{j=0}^{m-1}r_j$. It must be the case that $\frac{1}{N}\leq\gamma_1$. Otherwise, since $\gamma_1$ is the maximum sum of the weights $p_{ij}$ in a column, we would have $p_{ij}<\frac{1}{N}$ for all $i,j$, and it would follow that $\sum p_{ij}<1$, contrary to the fact that $\sum p_{ij}=1$.

Now define $f(x)$ on $(0,\infty)$ by $$f(x)=\p{\sum_{j=0}^{m-1}r_j^x}^\frac{1}{x}=\p{\sum_{r_j\neq0}r_j^x}^\frac{1}{x}.$$

Note that $f(x)$ is positive and continuous. By using logarithmic differentiation, we see that

\begin{align*}\frac{df}{dx}=&f(x)\cdot\p{-\frac{1}{x^2}\ln\p{\sum_{r_j\neq0}r_j^x}+\frac{1}{x}\frac{\sum_{r_j\neq0}\ln(r_j)r_j^x}{\sum_{r_j\neq0}r_j^x}}\\
=&f(x)\cdot\frac{x\sum_{r_j\neq0}\ln(r_j)r_j^x-\ln\p{\sum_{r_j\neq0}r_j^x}\sum_{r_j\neq0}r_j^x}{x^2\sum_{r_j\neq0}r_j^x}\\
=&f(x)\cdot\frac{\sum_{r_j\neq0}\p{x\ln(r_j)-\ln\p{\sum_{r_k\neq0}r_k^x}}r_j^x}{x^2\sum_{r_j\neq0}r_j^x}.
\end{align*}

Note that $f(x)$ is positive, the denominator is positive, and each $r_j^x$ is nonnegative. Furthermore, for each $j$, we have
\begin{align*}x\ln(r_j)-\ln\p{\sum_{r_k\neq0}r_k^x}&\leq x\ln(r_j)-\ln(r_j^x)=0.\end{align*}

It follows that $\frac{df}{dx}$ is nonpositive on $(0,\infty)$, and hence that $f(x)$ is nonincreasing on $(0,\infty)$. Thus, since $\frac{\ln(m)}{\ln(n)}<1$, we have
$$N=\sum_{j=0}^{m-1}r_j=f(1)\leq f\p{\frac{\ln(m)}{\ln(n)}}=\xi_1.$$

Thus, $\frac{1}{\xi_1}\leq \frac{1}{N}\leq\gamma_1$, as desired. 

Now we will establish that $\frac{1}{\xi_2}\leq\gamma_2$.  Let $M=\abs{\set{j:r_j>0}}$. We must have $\gamma_2\geq\frac{1}{M}$, for otherwise $\sum p_{ij}=\sum_{r_j> 0}\sum_{i=0}^{n-1}p_{ij}<\sum_{r_j>0}\frac{1}{M}=1$, contrary to the fact that $\sum p_{ij}=1.$ At the same time, it is clear that $\xi_2\geq M$. Hence, $\frac{1}{\xi_2}\leq\frac{1}{M}\leq\gamma_2$, as desired.

By hypothesis, we have that $\gamma_1<\frac{n}{\xi_1}$ or $\gamma_2<\frac{m}{\xi_2}$.

First, suppose $\gamma_1<\frac{n}{\xi_1}$. Then

\begin{align*}\frac{1}{\xi_1}&\leq\gamma_1<\frac{n}{\xi_1}\\
1&\leq\xi_1\gamma_1<n\\
\frac{1}{\gamma_1}&\leq\xi_1<\frac{n}{\gamma_1}\\
-\ln(\gamma_1)&\leq\frac{\ln(n)}{\ln(m)}\ln\p{\sum_{j=0}^{m-1}r_j^\frac{\ln(m)}{\ln(n)}}<-\ln(\gamma_1)+\ln(n)\\
-\frac{\ln(\gamma_1)}{\ln(n)}&\leq\frac{\ln\p{\ds\sum_{j=0}^{m-1}r_j^\frac{\ln(m)}{\ln(n)}}}{\ln(m)}<-\frac{\ln(\gamma_1)}{\ln(n)}+1\\
\alpha&\leq\dim_{H}(K)<\alpha+1.
\end{align*}
If $\mu_1$ is singular, Theorem \ref{Thm-A} applies to $\mu_1$, showing that the measure $\mu$ admits Fourier series of type $\eqref{SingMarg}$. The only way $\mu_1$ can be Lebesgue measure is if the total weight in every column is $\frac{1}{n}$, implying $\gamma_1=\frac{1}{n}$, and so $\alpha=1$. This implies $\dim_{H}(K)<2$, so that $K$ has Lebesgue measure 0. Thus $\mu$ is singular and Theorem \ref{SingWithLebMarg} applies, and we get Fourier series of type $\eqref{LebMarg}$.

Now suppose $\gamma_2<\frac{m}{\xi_2}$. Then

\begin{align*}\frac{1}{\xi_2}&\leq\gamma_2<\frac{m}{\xi_2}\\
1&\leq\xi_2\gamma_2<m\\
\frac{1}{\gamma_2}&\leq\xi_2<\frac{m}{\gamma_2}\\
-\ln(\gamma_2)&\leq\ln\p{\sum_{j=0}^{m-1}r_j^\frac{\ln(m)}{\ln(n)}}<-\ln(\gamma_2)+\ln(m)\\
-\frac{\ln(\gamma_2)}{\ln(m)}&\leq\frac{\ln\p{\ds\sum_{j=0}^{m-1}r_j^\frac{\ln(m)}{\ln(n)}}}{\ln(m)}<-\frac{\ln(\gamma_2)}{\ln(m)}+1\\
\beta&\leq\dim_{H}(K)<\beta+1.
\end{align*}

By an argument symmetric to the previous one, we get Fourier series of type $\eqref{SingMarg}$ if $\mu_2$ is singular, and of type $\eqref{LebMarg}$ if $\mu_2$ is Lebesgue.
\end{proof}

\begin{note}If $2\leq m\leq n$, then by exchanging the roles of $x$ and $y$ and the rows and columns, we arrive at a similar result.
\end{note}

\begin{corollary}[Square Grid IFS's]\label{Cor:SquareCarpet} Let $K$ be a Bedford-McMullen carpet generated by an $m\times m$ matrix $T$, and let $\mu$ be the invariant measure of the associated IFS $$\set{\tau_{ij}(x,y)=(\frac{x}{m}+\frac{i}{m},\frac{y}{m}+\frac{j}{m}):t_{ij}=1}$$ with positive weights $p_{ij}$. Let $N$ be the number of $1$'s in $T$, and let $\ds\gamma=\max_{i}\set{\sum_{\substack{j\\t_{ij}=1}}p_{ij}}$ be the maximum total weight of the IFS in a column. If $\gamma<\frac{m}{N}$, then $L^2(\mu)$ admits sequential Fourier expansions of type $\eqref{SingMarg}$ or $\eqref{LebMarg}$ according to whether $\mu_1$ is singular or Lebesgue.
\end{corollary}

\begin{proof}If $m=n=1$, then $\mu$ is Lebesgue measure and the result is obvious. Otherwise, apply Corollary \ref{Cor-A} with $n=m\geq 2$.
\end{proof}

\begin{example}[Measures on the Sierpinski Carpet]Consider the Bedford-McMullen carpet $K$ generated by the IFS corresponding to the matrix \begin{equation*}T=\begin{bmatrix}1&1&1\\1&0&1\\1&1&1\end{bmatrix}.\end{equation*} Then $K$ is the familiar Sierpinski carpet, also called the Menger Mesh.

It would first behoove us to consider the most popular measure associated to this IFS, the invariant measure $\mu$ obtained when all eight kept cells are weighted equally. That is, the weights $p_{ij}$ are given by the entries of \begin{equation*}P=\begin{bmatrix}\frac{1}{8}&\frac{1}{8}&\frac{1}{8}\\\frac{1}{8}&0&\frac{1}{8}\\\frac{1}{8}&\frac{1}{8}&\frac{1}{8}\end{bmatrix},\end{equation*} where $P$ is indexed as $T$ is. We observe that the projection of the IFS to either axis results in weights $\set{\frac{3}{8},\frac{2}{8},\frac{3}{8}}$. Thus both marginal measures $\mu_1$ and $\mu_2$ are the same, and since the projected weights are unequal, by Theorem $\ref{Thm:KakutaniIFS}$, both $\mu_1$ and $\mu_2$ are singular. Thus by Corollary \ref{Cor_Collect}, $\mu$ is bi-slice-singular and admits sequential Fourier expansions of type $\eqref{SingMarg}$.

Now, suppose we use a different weight set. Let instead
$$P=\begin{bmatrix}\frac{91}{900}&\frac{118}{900}&\frac{91}{900}\\\frac{150}{900}&0&\frac{150}{900}\\\frac{91}{900}&\frac{118}{900}&\frac{91}{900}\end{bmatrix}.$$
Let $\nu$ be the invariant measure associated to the IFS with these weights. To understand this measure, we display the cells and the probability associated to them after one, two, three, and four iterations of this IFS:

\begin{figure}[H]\includegraphics[scale=.55]{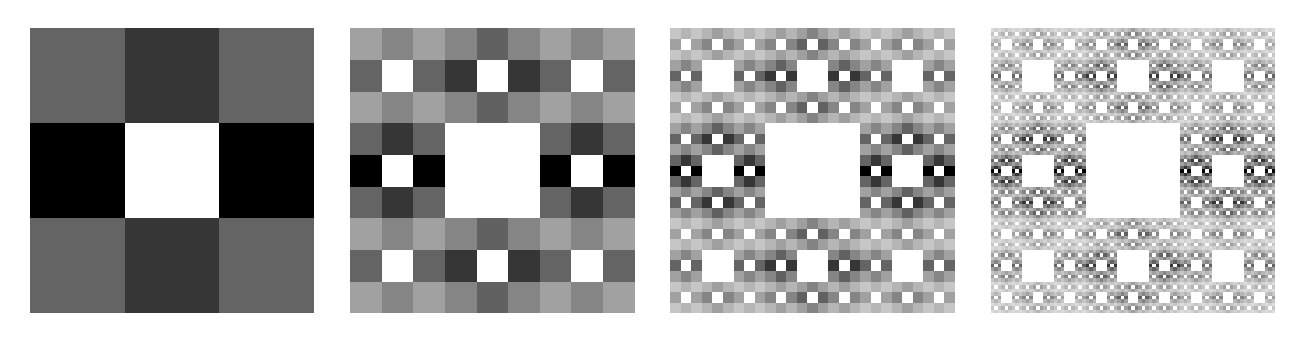}\caption{Sierpinski Carpet IFS with Some Unequal Weights}\end{figure}

One can see the Sierpinski carpet emerging as the attractor $K$ of the IFS, but because the weights are unequal, at each iteration some of the remaining parts have greater probability than others. In these images white represents zero probability, while cells with the greatest probability in that iteration are in black, and proportionally intense grays represent probabilities in between. The measure $\nu$ will be supported on the Sierpinski carpet just as $\mu$ is, but it will be a measure singular not only with respect to Lebesgue measure, but by the Kakutani Theorem it will also be singular with respect to $\mu$.

Examining $P$, we see that projected to the $y$-axis, the weights are $\set{\frac{1}{3},\frac{1}{3},\frac{1}{3}}$. Thus $\nu_2$ is Lebesgue measure. Projected to the $x$-axis, the weights are $\set{\frac{332}{900},\frac{236}{900},\frac{332}{900}}$, and thus $\nu_1$ is singular. By Corollary \ref{Cor_Collect}, $\nu$ admits Fourier series of type $\eqref{LebMarg}$. However, by using Corollary $\ref{Cor:SquareCarpet}$ we can say more. Note that $K$ comes from a square grid IFS with $m=3$ and $N=8$. The largest total weight in a column is $\gamma=\frac{91}{900}+\frac{150}{900}+\frac{91}{900}=\frac{332}{900}<\frac{3}{8}=\frac{m}{N}.$ Since $\nu_1$ is singular, by Corollary \ref{Cor:SquareCarpet} we see that $\nu$ admits Fourier expansions of type $\eqref{SingMarg}$. Indeed, the corollary shows that $\nu$ is $x$-slice-singular, though because $\nu_2$ is Lebesgue measure it is not bi-slice-singular.

\end{example}

Corollary \ref{Cor-A} does not apply in the case that all the entries of $T$ are 1, i.e. that the IFS keeps all the cells, because it that case the needed conditions reduce to $\gamma_1<\frac{1}{n}$ or $\gamma_2<\frac{1}{m}$, each of which would imply the weights add up to less than 1, which is a contradiction. The following result helps us address that case.

\begin{theorem}\label{Thm-fullgrid}Let $\mu$ be the invariant measure of a grid IFS that keeps all cells with nonzero weights. Suppose $\mu$ has a singular $y$-marginal ($x$-marginal) measure. Let $\beta_{\max}$ denote the maximum total weight in a row (column), and let $b$ denote the number of rows (columns) within which the weights are equal. If $\beta_{\max}<\frac{1}{b}$ (or if $b=0$), then $\mu$ is slice singular and hence admits Fourier expansions of form \eqref{SingMarg}.
\end{theorem}

\begin{proof}Consider a grid IFS generated by an $m\times n$ matrix $T$, where without loss of generality we suppose that $2\leq m\leq n$ and the $y$-marginal measure $\mu_2$ is singular. Let $p_{ij}$ be the weights of this IFS. For each row $j$, define $\beta_j=\sum_{i=0}^{n-1}p_{ij}$ to be the total weight in row $j$.

Now, fix a $y\in[0,1)$ that is non-$m$-adic. Then $y$ has a unique $m$-adic expansion $y=\sum_{k=0}^{\infty}\frac{c_k}{m^{k+1}}$, where each $c_k\in\set{0,1,\ldots,m-1}.$ Define a measure $\nu^y$ on $\set{0,\ldots,n-1}^{\mathbb{N}_0}$ via the Kolmogorov Extension Theorem by specifying its value on cylinder sets as follows:
\begin{align*}&\nu^y\p{\set{a_0}\times\set{a_1}\times\set{a_2}\times\cdots\times\set{a_\ell}\times\set{0,\ldots,n-1}^{\mathbb{N}_0}}\\
&=\frac{p_{a_0c_0}}{\beta_{c_0}}\cdot\frac{p_{a_1c_1}}{\beta_{c_1}}\cdot\frac{p_{a_2c_2}}{\beta_{c_2}}\cdots\cdot\frac{p_{a_\ell c_\ell}}{\beta_{c_\ell}}.
\end{align*}
Now define a measure $r^y$ on $[0,1)$ by $r(A)=\nu^y(\Pi_n^{-1}(A)),$ where
\begin{align*}\Pi_n:\set{0,\ldots,n-1}^{\mathbb{N}_0}\to\mathbb{R}:\omega\mapsto\sum_{k=0}^{\infty}\frac{\omega_k}{n^{k+1}}.
\end{align*}
It may be that for some rows $j$, the weight set $\set{p_{ij}}_{i=0}^{n-1}$ consists of equal weights, that is, $p_{ij}=\frac{\beta_j}{n}$ for all $i=0,\ldots n-1$.  Let $R$ denote the rows that do not consist of all equal weights, that is, $$R:=\set{j\in\set{0,\ldots,m-1}:\exists i,i^\prime, p_{ij}\neq p_{i^\prime j}}.$$ Also let $R^\mathcal{C}$ denote the set $\set{0,\ldots,m-1}\setminus R$.

Let $m$ denote the uniform measure on $\set{0,\ldots, n-1}$, and for each $k$, let $m_k$ denote the measure on $\set{0,\ldots,n-1}$ given by $m_k(\set{i})=\frac{p_{ic_k}}{\beta_{c_k}}$. As in the Kakutani Theorem, denote 
\begin{align*}\rho(m,m_k)=\int \sqrt{\frac{dm_k}{dm}}\,dm=\sum_{i=0}^{n-1}\sqrt{m_k(\set{i})m(\set{i}).}
\end{align*}
We have that for $c_k\in R$, $m_k\neq m$, and hence $\rho(m,m_k)<1$, whereas for $c_k\in R^\mathcal{C}$, $m_k=m$ and $\rho(m,m_k)=1$.

If $R$ is nonempty, then since $R$ is finite, $\max\set{\rho(m,m_k):c_k\in R}<1.$ Therefore, so long as infinitely many terms of the sequence $\set{c_k}$ are in $R$, we must have that \begin{align*}\prod_{k=0}^{\infty}\rho(m,m_k)=0.
\end{align*}
Then by the Kakutani Theorem, the measure $\nu^y=\times_{k=0}^{\infty}m_k$ is perpendicular to the uniform measure $\times_{k=0}^{\infty}m$, and hence $r^y$ is perpendicular to the Lebesgue measure on $[0,1).$ Otherwise, if $\set{c_k}$ has only finitely many terms in $R$, or what is equivalent, a tail consisting entirely of elements of $R^\mathcal{C}$, then clearly the product will not be $0$ and $r^y$ will be equivalent to Lebesgue measure. If $R$ is empty, that is to say, every row $j$ has weights that are equal within that row, then every sequence $\set{c_k}$ consists entirely of elements of $R^\mathcal{C}$, and $m_k=m$ for each $k$, and the measure $r^y$ is Lebesgue measure for any non-$m$-adic $y\in[0,1)$.

Now let $Y$ denote the set of all sequences in $\set{0,\ldots,m-1}^{\mathbb{N}_0}$ with tails in $R^\mathcal{C}$. By the preceding discussion, for a non-$m$-adic $y$, $r^y$ is equivalent to Lebesgue measure if $y\in \Pi_m(Y)$ and singular otherwise. 

Now, the $y$-marginal measure $\mu_2$ cannot have any atoms unless some row $j$ has full weight, that is, $\beta_j=1$. In this case, $\mu_2$ consists of a single atom, and the slice measure at this atom is either singular or Lebesgue according to the equality or non-equality of the weights in that row. However, this is impossible for a grid IFS in which all cells are kept with nonzero weights. Hence, $\mu_2$ has no atoms (and by the same logic, the $x$-marginal measure $\mu_1$ has no atoms either). So the set of $m$-adic numbers, being countable, has $\mu_2$ measure $0$. Thus by establishing that $\mu_2(\Pi_m(Y))=0$, we will have shown that $r^y$ is singular for $\mu_2$-almost-every-$y$.

Consider the case $b=0$. (In this case, by convention we may say $\beta_{\max}<\infty=\frac{1}{b}$.) We have $\abs{R^\mathcal{C}}=b=0$. So $R^\mathcal{C}$ is empty, implying $Y$ is empty, and so trivially, $r^y$ is singular for $\mu_2$-almost-every $y$.

In the case $b>0$, we have that $R^\mathcal{C}$ is not empty. We claim that $\dim_H(\Pi_m(Y))=\frac{\ln(b)}{\ln(m)}.$ Let $W$ denote the set of all finite words in $\set{0,\ldots,m-1}$ terminating in an element of $R$. For each $w\in W$, let $Y_w$ be the subset of $Y$ consisting of sequences that begin with $w$, followed by terms in $R^\mathcal{C}$. It is clear that the sets $Y_w$ are pairwise disjoint and that $Y=\cup_{w\in W}Y_w$, from which it follows that $\Pi_m(Y)=\cup_{w\in W}\Pi_m(Y_w)$. Since $W$ is countable, we have \begin{align*}\dim_{H}(\Pi_m(Y))=\sup_{w\in W}\dim_{H}(\Pi_m(Y_w)).\end{align*}
It thus suffices to show that for all $w\in W$, $\dim_{H}(\Pi_m(Y_w))=\frac{\ln(b)}{\ln(m)}.$ 

Define $\tau_j(y)=\frac{y}{m}+\frac{j}{m}$, and consider the IFS $\set{\tau_j:j\in R^\mathcal{C}}$. Let $K$ be the attractor of this IFS. We know that \begin{align*}\dim_{H}(K)=\frac{\ln(\abs{R^\mathcal{C}})}{\ln(m)}=\frac{\ln(b)}{\ln(m)}.\end{align*}
Now for a fixed $w\in W$, with $w=w_0w_1\ldots w_\ell$, observe that
\begin{align*}\Pi_m(Y_w)=\tau_{w_0}\circ\tau_{w_1}\circ\cdots\circ \tau_{w_\ell}(K).
\end{align*}
The map $\tau_{w_0}\circ\tau_{w_1}\circ\cdots\circ \tau_{w_\ell}:\mathbb{R}\to\mathbb{R}$ is bi-Lipschitz and therefore preserves Hausdorff dimension. Thus $\dim_{H}(\Pi_m(Y_w))=\frac{\ln(b)}{\ln(m)}$, and so \begin{align*}\dim_{H}(\Pi_m(Y))=\frac{\ln(b)}{\ln(m)}.\end{align*}

Now, let $\beta_{\max}=\max\set{\beta_j:j\in\set{0,\ldots,m-1}}$. Following the previously used technique, we see that for any interval $I$, $\mu_2(I)\leq (m+1)\abs{I}^\alpha$, where $\alpha=-\frac{\ln(\beta_{\max})}{\ln(m)}$. If $\mu_2(\Pi_m(Y))>0$, it would follow from the Mass Distribution Principle (see \cite{BisPer2017}, Lemma 1.2.8) that \begin{align*}\dim_{H}(\Pi_m(Y))\geq-\frac{\ln(\beta_{\max})}{\ln(m)}.\end{align*} So by the contrapositive, 
\begin{align*}\frac{\ln(b)}{\ln(m)}<-\frac{\ln(\beta_{\max})}{\ln(m)}\implies\mu_2(\Pi_m(Y))=0.\end{align*}
Equivalently,
\begin{align*}\beta_{\max}<\frac{1}{b}\implies\mu_2(\Pi_m(Y))=0.\end{align*}

Hence, $r^y$ is singular for $\mu_2$-almost-every-$y$ when $\beta_{\max}<\frac{1}{b}$. The proof will be completed by establishing that the measures $r^y$ are $\mu_2$-almost-everywhere equal to the measures $\rho^y$ from the Rokhlin disintegration\\ $\mu=\rho^y(dx)\,\mu_2(dy)$.

Now consider an $(n,m)$-adic rectangle of the form
\begin{align*}R_\sigma=\left[x_0,x_0+\frac{1}{n^K}\right)\times\left[y_0,y_0+\frac{1}{m^{K}}\right),\end{align*}
where
\begin{align*}x_0&=\sum_{k=0}^{K-1}\frac{d_k}{n^{k+1}}\\
y_0&=\sum_{k=0}^{K-1}\frac{c_k}{m^{k+1}}.
\end{align*}
In light of the fact that $\mu$'s marginal measures have no atoms,
\begin{align*}\mu(R_\sigma)=\prod_{k=0}^{K-1}p_{d_kc_k}.
\end{align*}
At the same time, by observing that all $y\in[y_0,y_0+\frac{1}{m^K})$ share the same first $K$ $m$-adic digits, we see that

\begin{align*}&\int\int\chi_{R_\sigma}\,r^y(dx)\,d\mu_2(dy)\\
=&\int_{[y_0,y_0+1/m^K)}r^y([x_0,x_0+1/n^K))\,d\mu_2(dy)\\
=&\int_{[y_0,y_0+1/m^K)}\prod_{k=0}^{K-1}\frac{p_{d_kc_k}}{\beta_{c_k}}\,d\mu_2(dy)\\
=&\p{\prod_{k=0}^{K-1}\frac{p_{d_k c_k}}{\beta_{c_k}}}\mu_2([y_0,y_0+1/m^K))\\
=&\p{\prod_{k=0}^{K-1}\frac{p_{d_k c_k}}{\beta_{c_k}}}\p{\prod_{k=0}^{K-1}\sum_{j=0}^{n-1}p_{jc_k}}\\
=&\p{\prod_{k=0}^{K-1}\frac{p_{d_k c_k}}{\beta_{c_k}}}\p{\prod_{k=0}^{K-1}\beta_{c_k}}\\
=&\prod_{k=0}^{K-1}p_{d_k c_k}.
\end{align*}

Let $\mathcal{F}$ denote the collection of all finite unions of $(n,m)$-adic rectangles of the above form. Clearly, $\mathcal{F}$ is an algebra over $[0,1)^2$. Define $\kappa:[0,1)\times\mathcal{B}([0,1)^2)\to[0,\infty)$ by $$\kappa(y,A)=\int\chi_{A}(x,y)\,r^y(dx).$$ (For an $m$-adic $y$, where $\mu_2$ is unsupported, $r^y$ may without consequence be regarded as Lebesgue measure for the purpose of defining $\kappa$.) The previous computation shows that for an $(n,m)$-adic rectangle $R_\sigma$, 
\begin{align*}\kappa(y,R_\sigma)=\begin{cases}\prod_{k=0}^{K-1}p_{d_kc_k}&y\in[y_0,y_0+\frac{1}{m^K})\text{ and y non-m-adic}\\\frac{1}{n^K}&y\in[y_0,y_0+\frac{1}{m^K})\text{ and y m-adic}\\0&\text{otherwise}\end{cases}\end{align*}
is a simple function in $y$ and hence measurable in $y$. It follows that it is measurable for any $F\in\mathcal{F}$, and because $\mathcal{F}$ generates the Borel sigma algebra, by the Monotone Class Theorem $\kappa(y,A)$ is thus measurable in $y$ for any Borel set $A$ (see Appendix \ref{App:fullgrid}). Thus, $\kappa(y,A)$ is a Markov transition kernel, from which it follows that $r^y(dx)\,\mu_2(dy)$ defines a measure on $[0,1)^2$.

Since the $(n,m)$-adic rectangles of the above form are a $\pi$-system on $[0,1)^2$ that generates the Borel sigma algebra, and $\mu(dx\,dy)$ and $r^y(dx)\,\mu_2(dy)$ coincide on this system, we have that $\mu=r^y(dx)\,\mu_2(dy)$. By the uniqueness of the Rokhlin disintegration, we then have that $r^y=\rho^y$ for $\mu_2$-almost-every $y$, which completes the proof. 
\end{proof}

\begin{example}Consider the grid IFS and weight set given by \begin{equation}T=\begin{bmatrix}1&1&1&1\\1&1&1&1\\1&1&1&1\\1&1&1&1\end{bmatrix}, P=\begin{bmatrix}\frac{3}{32}&\frac{3}{32}&\frac{3}{32}&\frac{3}{32}\\\frac{1}{32}&\frac{1}{32}&\frac{2}{32}&\frac{2}{32}\\\frac{2}{32}&\frac{2}{32}&\frac{1}{32}&\frac{1}{32}\\\frac{2}{32}&\frac{2}{32}&\frac{2}{32}&\frac{2}{32}\end{bmatrix}.\end{equation}
This is a full grid IFS, with every cell kept. The attractor is the entire unit square. However, with the weight set being unequal, the invariant measure $\mu$ will be singular by Theorem \ref{Thm:KakutaniIFS}. We give a graphical representation of the first three iterations of this IFS with these weights:

\begin{figure}[H]\includegraphics[scale=.55]{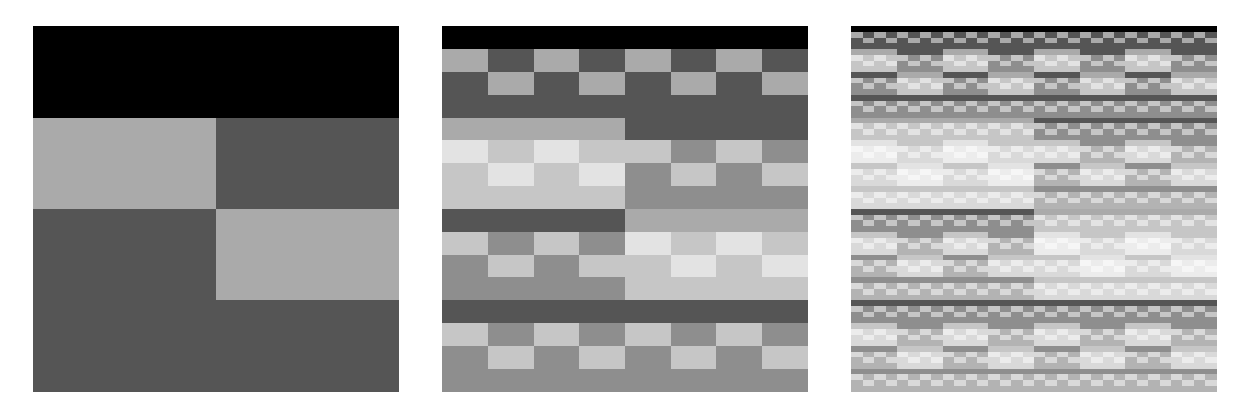}\caption{A full grid IFS}
\end{figure}

Note that the projection projection of the IFS to the $x$-axis is an IFS with weight set $\set{\frac{1}{4},\frac{1}{4},\frac{1}{4},\frac{1}{4}}$. Thus the $x$-marginal measure $\mu_1$ is Lebesgue measure. We thus have Fourier series of type $\eqref{LebMarg}$ by Corollary \ref{Cor_Collect}, but we can say more.

The projection of the IFS to the $y$-axis is an IFS with weight set $\set{\frac{8}{32},\frac{6}{32},\frac{6}{32},\frac{12}{32}}$. We thus see that $\mu_2$ is a singular measure, and the maximum total weight in a row is $\beta_{\max}=\frac{12}{32}.$ Further, we see that two of the rows have equal weights within themselves, so $b=2$. Thus $\beta_{\max}=\frac{12}{32}<\frac{1}{2}=\frac{1}{b}$. Theorem \ref{Thm-fullgrid} thus implies that the measure is $y$-slice-singular and hence has Fourier series of type \eqref{SingMarg}.
\end{example}

\section{Appendix}

\subsection{Theorem \ref{Thm-fullgrid} Detail}\label{App:fullgrid}We add some detail to the proof of Theorem \ref{Thm-fullgrid}. Let $\mathcal{B}$ denote the Borel sigma algebra on $[0,1)^2$. Recall that we defined\\ $\kappa:[0,1)\times\mathcal{B}\to\mathbb{R}$ by \begin{align*}\kappa(y,A)=\int\chi_{A}(x,y)\,r^y(dx),\end{align*}
where for each $y\in[0,1)$, $r^y$ was a Borel probability measure defined in a particular way from the grid IFS. We also showed that there exists an algebra $\mathcal{F}\subset\mathcal{B}$ generating $\mathcal{B}$ such that for any $F\in\mathcal{F}$, $\kappa(y,F)$ is measurable in the variable $y$.

Now define $\mathcal{H}=\set{X\in\mathcal{B}:\kappa(y,X)\text{ is measurable in }y}$. Clearly, $\mathcal{F}\subseteq\mathcal{H}$. It is evident from its definition that if $X\subseteq Y$, then $\kappa(y,X)\leq\kappa(y,Y)$. Therefore, if $X_1\subseteq X_2\subseteq X_3\subseteq\cdots$ is an ascending monotone sequence in $\mathcal{H}$, $(\kappa(y,X_n))_n$ is a bounded, monotone increasing sequence of functions, which therefore converges pointwise in the variable $y$. Indeed, by lower continuitiy of measure, $(\kappa(y,X_n))\to\kappa(y,\cup X_n)$ pointwise. Since a pointwise limit of meaurable functions is measurable, it follows that $\kappa(y,\cup X_n)$ is measurable, and therefore that $\cup X_n \in\mathcal{H}$. So $\mathcal{H}$ is closed under countable monotone unions, and by a similar argument, it is closed under countable monotone intersections. So, $\mathcal{H}$ is a monotone class.

By the Monotone Class Theorem for Sets, $\mathcal{B}$ is the smallest monotone class containing $\mathcal{F}$. Thus $\mathcal{B}\subseteq\mathcal{H}$, but also $\mathcal{H}\subseteq\mathcal{B}$ by definition, so in fact $\mathcal{H}=\mathcal{B}$. This proves that $\kappa(y,A)$ is measurable in $y$ for any Borel set $A\in\mathcal{B}$.

\subsection{Invariant Measures and IFS Weights}
\label{App:SymbolicRepresentations}
An IFS combined with a set of weights gives rise to an invariant measure $\mu$, but the weights only indirectly determine the values $\mu$ takes. We explain here. A more detailed treatment of this construction can be found in \cite{BarHocRap19}, and it also appears in a nice series of lectures by Solomyak given at Bar-Ilan in 2018.

Suppose $\set{\tau_{1},\tau_{2},\ldots,\tau_{N}}$ is an IFS on a complete metric space $X$. Let $(\omega_k)$ be a sequence in $(\mathbb{Z}_N)^{\mathbb{N}_0}.$ Then $(\omega_k)$ is a \textbf{symbolic address} of the element $x$ given by \begin{equation}x=\lim_{k\to\infty}\tau_{\omega_k}\circ\tau_{\omega_{k-1}}\circ\cdots\circ\tau_{\omega_0}(x_0).\end{equation}

It does not matter what $x_0$ is; the limit will converge to $x$ regardless. Thus we can define $\Pi:(\mathbb{Z}_N)^{\mathbb{N}_0}\to X$ by \begin{equation}\Pi(\omega)=\lim_{k\to\infty}\tau_{\omega_k}\circ\tau_{\omega_{k-1}}\circ\cdots\circ\tau_{\omega_0}(x_0).\end{equation}

The set of all symbolic addresses of $x$ is then the preimage $\Pi^{-1}(\set{x})$. Essentially, a sequence of contractions corresponds to a unique point of the metric space, though there can be multiple sequences corresopnding to the same point.

Now, suppose $\set{w_1,\ldots,w_N}$ is a set of nonzero weights with $\sum w_k=1$. These weights can be used to define a Bernoulli product measure $\nu$ on $(\mathbb{Z}_N)^{\mathbb{N}_0}$ by defining $\nu$ on cylinder sets as \begin{equation}\nu\p{\set{\omega_0}\times\set{\omega_1}\times\cdots\times\set{\omega_k}\times(\mathbb{Z}_N)^{\mathbb{N}_0}}=\prod_{\ell=0}^{k}w_{\omega_\ell}\end{equation}
and then extending it to the whole of $(\mathbb{Z}_N)^{\mathbb{N}_0}$ via the Kolmogorov Extension Theorem.

Finally, the invariant measure $\mu$ for this IFS with these weights is given by \begin{equation}\mu(X)=\nu(\Pi^{-1}(X)).\end{equation}

Now, in the situation of a two-dimensional grid IFS, the metric space is $\mathbb{R}^2$, and we have an $m\times n$ matrix $T$ such that the contractions are given by 
\begin{equation*}\set{\tau_{ij}(x,y)=\p{\frac{x}{n}+\frac{i}{n},\frac{y}{m}+\frac{j}{m}}:t_{ij}=1}\end{equation*}
and a set of weights given by $P=(p_{ij})$.

We can then let \begin{equation}\Omega=\set{(i,j):t_{ij}=1}^{\mathbb{N}_0}\subseteq(\mathbb{Z}_n\times\mathbb{Z}_m)^{\mathbb{N}_0}.\end{equation}

We define the Bernoulli product measure $\nu$ on $\Omega$ by defining it on cylinder sets as
\begin{equation}\nu\p{\set{x_0,y_0}\times\set{x_1,y_1}\times\cdots\times\set{x_k,y_k}\times\Omega}=\prod_{\ell=0}^{k}p_{x_\ell y_\ell}\end{equation}
and then extending it via the Kolmogorov Extension Theorem. We define $\Pi:\Omega\to\mathbb{R}^2$ by \begin{align*}\Pi((x_\ell,y_\ell)_\ell)=&\lim_{k\to\infty}\tau_{x_ky_k}\circ\tau_{x_{k-1}y_{k-1}}\circ\cdots\circ\tau_{x_0y_0}(0,0)\\
=&\p{\sum_{k=0}^{\infty}\frac{x_k}{n^{k+1}},\sum_{k=0}^{\infty}\frac{y_k}{m^{k+1}}}.\end{align*}
The invariant measure $\mu$ will satisfy \begin{equation}\mu(X)=\nu(\Pi^{-1}(X)).\end{equation}

 It is due to the fact that $m$-adic (or $n$-adic) numbers have two expansions, one that ends in trailing $0$'s for its digits and another that ends in trailing $(m-1)$'s, that $\Pi$ is not injective, leading to the potential difference between the weights $p_{ij}$ and the value of $\mu$ on closed $(n,m)$-adic rectangles.

\bibliographystyle{plain} 
\bibliography{bib2024} 
\end{document}